\documentclass[11pt,a4paper,oneside,reqno]{amsart}

\usepackage{paquet_en}

\title{Geodesic simplices of pseudo-hyperbolic space}
\author[T. Lemistre]{Timoth\'e Lemistre}
\address{LJAD, Universit\'e C\^ote d'Azur, Nice, France}
\email{timothe.lemistre@proton.me}
\date{\today}

\begin{document}
	
\begin{abstract}
	We give a cohomological interpretation of the geodesic simplices of the pseudo-hyperbolic space of signature $(p,q)$ and formulate a necessary and sufficient condition for such a simplex to have finite volume. As a corollary, we obtain that every ideal geodesic polytope in the pseudo-hyperbolic space of signature $(2,2)$ has finite volume.
\end{abstract}

\maketitle
{\setlength{\parskip}{0pt}
\tableofcontents}

\section*{Introduction}

A central fact in hyperbolic geometry is that the geodesic simplices of the hyperbolic space $\H^p$ have volume bounded above by a constant $\bf{v}(\H^p)$ depending only on $p$. Moreover, the maximal volume is attained only for the regular ideal geodesic simplex in $\H^p$ (result of Haagerup-Munkholm \cite{HM81} for $p \geq 4$). This fact is for example crucial in Gromov's foundational article on simplicial volume and bounded cohomology \cite{Gromov82}, in which he proves that for every closed hyperbolic $p$-manifold $M$, denoting by $\Vol(M)$ its volume and by $\norm{M}$ its simplicial volume, one has
\[\frac{\Vol(M)}{\norm{M}} = \bf{v}(\H^p). \]
He deduces from this an elegant proof of Mostow’s rigidity theorem for hyperbolic $p$-manifolds, using rigidity results of Thurston \cite{Thurston1980} and the characterisation of the geodesic simplices of maximal volume.

\textbf{Pseudo-hyperbolic space.}
In this article, we study geodesic simplices in the pseudo-hyperbolic space of signature $(p,q)$, denoted by $\H^{p,q}$. The latter is a pseudo-Riemannian manifold of constant curvature that generalises the hyperbolic space $\H^p$. The recent interest in this space comes from its central role in higher-rank Teichmüller theory (see in particular \cite{Mess90}, \cite{Barbot13}, \cite{BM12}, \cite{DGK18}, \cite{BK25}, \cite{BS10} et \cite{BS20}).

Let us call geodesic simplex, or more simply \emph{simplex} of $\H^{p,q}$ a set of $p+q+1$ points in $\P(\R^{p,q+1})$, spanning $\R^{p,q+1}$ and whose interior of the convex hull, in some well-chosen affine chart, is contained in $\H^{p,q}$. Such a convex hull is determined up to isometry by the simplex. A simplex is said to be \emph{ideal} if its points are in the projective boundary of $\H^{p,q}$. We deduce from our main theorem the following corollary (Proposition \ref{proposition:existence_simplexes_infinis} and Corollary \ref{corollaire:simplexe_H22_fini}). 

\textbf{\hyperref[corollaire:simplexe_H22_fini]{Main corollary}}. \textit{There exist simplices of infinite volume in $\H^{2,2}$ but none is ideal. For all integers $p \geq 3$ and $q \geq 1$, there exist simplices of infinite volume in $\H^{p,q}$, both ideal and non-ideal. }

Our main theorem is a finiteness criterion characterising simplices of infinite volume in $\H^{p,q}$. Given a simplex $\Lambda$ of $\H^{p,q}$, we define the undirected graph $\cal{G}(\Lambda)$ having $\Lambda$ as set of vertices and whose edges are the $\{\xi,\eta\}$ such that the lines $\xi$ and $\eta$ are not orthogonal. We can then read properties of $\Lambda$ on the graph $\cal{G}(\Lambda)$ : let us illustrate this. Given a subset $I$ of the set of vertices of a graph, denote by $\partial I$ the set of vertices that are not in $I$ but adjacent to a vertex of $I$. $I$ is said to be \emph{stable} if there is no edge between vertices of $I$.

\textbf{\hyperref[thm_principal]{Main theorem}}. \textit{Let $\Lambda$ be a simplex of $\H^{p,q}$. $\Lambda$ has infinite volume if, and only if, the graph $\cal{G}(\Lambda)$ contains a non-empty, stable set of vertices $I$ with cardinality at least $\Card \partial I$. }

We introduce a second tool : an interpretation of the isometry class of a simplex $\Lambda$ as an element of a (twisted) cohomology group associated to the graph $\cal{G}(\Lambda)$. This allows more generally to encode isometry classes of subsets of $\H^{p,q}$, which leaves hope for further applications.

\textbf{Organisation.} Section \ref{section:préliminaires} contains preliminaries of pseudo-hyperbolic geometry. Section \ref{section:classification} defines the cohomological interpretation and derives some consequences used in Section \ref{section:critère}. In Section \ref{section:critère}, we prove the \hyperref[thm_principal]{main theorem} and derive consequences from it. We show for example that every graph corresponding to an ideal simplex of $\H^{2,2}$ contains a $5$-cycle and deduce from this the \hyperref[corollaire:simplexe_H22_fini]{main corollary}.

\textbf{Acknowledgements.} I would like to express my gratitude to my advisors, Francesco Bonsante and Jérémy Toulisse, for their support and patience. I am also indebted to Enrico Trebeschi, Arthur Mazeyrat, Andrea Tamburelli and Pierre-Louis Blayac for interesting discussions about the content of this article and for their valuable comments.

\section{Preliminaries}\label{section:préliminaires}

From now on, let $p > 0$ and $q \geq 0$ be integers. We denote by $\R^{p,q+1}$ the oriented vector space $\R^{p+q+1}$ endowed with a non-degenerate symmetric bilinear form $\bf{b}$ of signature $(p,q+1)$, and by $\O(p,q+1)$ its orthogonal group. The group $\PO(p,q+1) = \O(p,q+1)/\{\pm \id\}$ acts on the projective space $\P(\R^{p,q+1})$ preserving an open set, the \emph{pseudo-hyperbolic space}, defined by
\[\H^{p,q} = \{[x] \in \P(\R^{p,q+1}) \mid \bf{b}(x,x) < 0\}. \]
The stabiliser of a point of $\H^{p,q}$ under the action of $\PO(p,q+1)$ is conjugate to $\PO(p,q)$. Define a double covering of $\H^{p,q}$ by
\[\Hc^{p,q} = \{x \in \R^{p,q+1} \mid \bf{b}(x,x) = -1\}. \]
We have $\Hc^{p,q}/{\pm \id} = \H^{p,q}$. A tangent space $\Tan_x \Hc^{p,q}$ is identified with $\Vect{x}^\perp$, hence inherits the metric $\bf{b}$ from $\R^{p,q+1}$. This endows $\Hc^{p,q}$ with a metric of signature $(p,q)$ which descends to a metric on $\H^{p,q}$. The metrics on $\Hc^{p,q}$ and $\H^{p,q}$ have constant curvature $-1$ and respective isometry groups $\O(p,q+1)$ and $\PO(p,q+1)$. Their respective Levi-Civita connections are geodesically complete, the geodesics of $\H^{p,q}$ being its inter"s with projective lines.

\textbf{Subsets of the projective space and graphs.}\label{définition:graphe}
By a \emph{graph} (with vertex set $\cal{S}$) we mean a pair $\cal{G} = (\cal{S},\cal{A})$ where $\cal{S}$ is a set (whose elements are called vertices) and $\cal{A}$ is a set of subsets of $\cal{S}$ of the form $\{\xi,\eta\}$ with $\xi$ and $\eta$ two vertices (possibly equal). The elements of $\cal{A}$ are called edges, and edges of the form $\{\xi,\xi\}$ are called loops. Two vertices $\xi,\eta$ (possibly equal) are said to be adjacent if $\{\xi,\eta\}$ is an edge. A morphism from a graph $(\cal{S}_0,\cal{A}_0)$ to a graph $(\cal{S}_1,\cal{A}_1)$ is a map $\phi : \cal{S}_0 \to \cal{S}_1$ such that for every edge ${\xi,\eta} \in \cal{A}_0$, one has $\{\phi(\xi),\phi(\eta)\} \in \cal{A}_1$.

In what follows, $\Lambda$ will be a subset $\Lambda$ of $\P(\R^{p,q})$, possibly infinite. To such a subset, we associate the graph $\cal{G}(\Lambda) = (\Lambda,\cal{A}(\Lambda))$ where $\cal{A}(\Lambda)$ is the set of edges $\{\xi,\eta\} \subset \Lambda$ such that the lines $\xi$ and $\eta$ are not orthogonal. We give examples of such graphs in Section \ref{subsection:exemples}. The set $\Lambda$ is said to be \emph{ideal} if $\cal{G}(\Lambda)$ has no loops, that is if $\Lambda$ is contained in the projective boundary of $\H^{p,q}$, denoted
\[\partial \H^{p,q} = \{[x] \in \P(\R^{p,q+1}) \mid \bf{b}(x,x) = 0\}. \]

\begin{definition*}
Let $\Lambda$ be a subset of $\P(\R^{p,q+1})$. We call \emph{lift} of $\Lambda$ the data, for every line $\xi \in \Lambda$, of a non-zero vector of $\xi$. In other words, a lift of $\Lambda$ is a section of the projectivisation $\P : \R^{p,q+1} \setminus \{0\} \to \P(\R^{p,q+1})$ over $\Lambda$. A lift $s$ is said \emph{\positif}\! if the convex hull $\Conv(s(\Lambda))$ does not contain $0$ and we have $\bf{b}(s(\xi),s(\eta)) \leq 0$ for all $\xi,\eta$ in $\Lambda$. If $s$ is \positif\!, then the projectivisation of $\Conv(s(\Lambda))$ is well-defined and its interior is an open subset of $\H^{p,q}$ which we denote $\Conv(s)$. 
\end{definition*}

\begin{lemme}
If $\Lambda$ admits a lift $s$ such that $\bf{b}(s(\xi),s(\eta)) \leq 0$ for all $\xi,\eta \in \Lambda$, then $\Lambda$ admits a \positif lift.
\end{lemme}
\begin{proof}
It suffices to modify $s$ so that $\Conv(s(\Lambda))$ does not contain $0$. Endow $\R^{p,q+1}$ with a strict order $\prec$ making it into an ordered vector space. For each isolated vertex $\xi$ of $\cal{G}(\Lambda)$, up to multiplying $s(\xi)$ by $-1$, we may assume that $0 \prec s(\xi)$. Suppose now that there exists a relation $0 = \sum_{i \in I} t_i s(\xi_i)$ with $t_i > 0$, $\xi_i \in \Lambda$, and $I$ a non-empty finite subset of $\Lambda$. Given an element $\eta$ of $\Lambda$, we have
\[0 = \bf{b}(s(\eta),\sum_{i \in I} t_i s(\xi_i)) = \sum_i t_i \bf{b}(s(\eta),s(\xi_i)). \]
Since $t_i > 0$ and $\bf{b}(s(\eta),s(\xi_i)) \leq 0$, it follows that $\bf{b}(s(\eta),s(\xi_i)) = 0$ for every $i$ : thus $\eta$ is not adjacent to any of the $\xi_i$. Hence the $\xi_i$ are isolated vertices of the graph $\cal{G}(\Lambda)$, so $0 \prec s(\xi_i)$, and therefore $0 \prec \sum_i t_i s(\xi_i) = 0$, which is absurd. Thus $0$ does not belong to $\Conv(s(\Lambda))$ and $s$ indeed defines a \positif lift.
\end{proof}

\textbf{Convex hulls.}\label{définition:enveloppe_convexe}
A \emph{convex hull} of a subset $\Lambda$ of $\P(\R^{p,q+1})$ is any open subset of $\H^{p,q}$ of the form $\Conv(s)$, where $s$ is a \positif lift of $\Lambda$.
\begin{lemme}[uniqueness of the convex hull]\label{lemme:relèvement_unique}
Assume that $\Vect{\Lambda}$ is non-degenerate. Two convex hulls of $\Lambda$ are identified by an element of $\PO(p,q+1)$.
\end{lemme}
\begin{proof}
Let $s_0$ and $s_1$ be two \positif lifts of $\Lambda$. There exists $g  : \Lambda \to \R^\Lambda$ such that $s' = g s$, of constant sign $\epsilon_i \in {\pm 1}$ on each of the connected components of the graph $\cal{G}(\Lambda)$, denoted by $\Lambda_i$ ($i \in I$). Let us show by contradiction that none of the $\Vect{\Lambda_i}$ can be degenerate : otherwise, we would have
\begin{align*}
	\Vect{\Lambda} &= \Vect{\Lambda_i} + \left( \sum_{j \neq i} \Vect{\Lambda_j} \right) \\
	&\subset \Vect{\Lambda_i} + \left( \Vect{\Lambda_i}^\perp \cap \Vect{\Lambda} \right) \\
	&\subsetneq \Vect{\Lambda} \tag*{since $\Vect{\Lambda_i}$ is degenerate,}
\end{align*}
which is a contradiction. Hence $\R^{p,q+1}$ is the orthogonal direct sum of $\Vect{\Lambda}^\perp$ and the $\Vect{\Lambda_i}$. We define an isometry of $\R^{p,q+1}$ equal to $\epsilon_i \id$ on each term $\Vect{\Lambda_i}$, and preserving $\Vect{\Lambda}^\perp$. It acts on $\P(\R^{p,q+1})$ by sending $\Conv(s_0)$ to $\Conv(s)$, which concludes the proof.
\end{proof}

\textbf{Polytopes.}\label{définition:polytope}
A \emph{polytope} (in $\H^{p,q}$) is a finite subset $\Lambda$ of $\P(\R^{p,q+1})$ admitting a \positif lift and such that each of its points is extremal in a convex hull of $\Lambda$ (ie. it cannot be written as a nontrivial convex combination of points $\Lambda$ ; it is then true in every convex hull by Lemma \ref{lemme:relèvement_unique}). The points of $\Lambda$ are called the vertices of the polytope. A \emph{marked polytope} is a bijection $m$ from $\cal{S}_n = \{1,\dots,n\}$ onto (the vertices of) a polytope. We say that $\phi \in \PO(p,q+1)$ defines an isometry between the marked polytopes $m$ and $m'$ if $m' = \phi \circ m$. A polytope is called a \emph{simplex} if it has $p+q+1$ vertices spanning the vector space $\R^{p,q+1}$. Lemma \ref{lemme:relèvement_unique} shows that given a simplex $\Lambda$, its \hyperref[définition:enveloppe_convexe]{convex hull} is unique up to isometry : its volume, called volume of $\Lambda$, is denoted by $\Vol(\Lambda)$. 

\section{Cohomological viewpoint}
\label{section:classification}

In this section, we define a correspondence between isometry classes of subsets of $\P(\R^{p,q})$ and (twisted) cohomology classes of their graph. The isometry classes of simplices of graph $\cal{G}$ then correspond to an open subset of a real vector space of finite dimension, whose dimension can be read on $\cal{G}$ (Lemma \ref{lemme:dimension_H1}). We deduce from this a correspondence between geometric properties of simplices and properties of their graph, which will be used in Section \ref{subsection:existence_simplexes}. 

\subsection{Definitions}

The idea is the following : a lift $s$ of a subset $\Lambda$ of $\P(\R^{p,q})$ defines a function associating to an edge $\{\xi,\eta\}$ of $\cal{G}(\Lambda)$ the real number $\bf{b}(s(\xi),s(\eta))$, which corresponds to the classical notion of Gram matrix. A changes of lift (which we will see as a $0$-cochain) modifies the Gram matrix (which we will see as a $1$-cochain). The set of Gram matrices modulo change of lift is then naturally defined as the cohomology of a very simple chain complex. We use here the vocabulary of cohomology in order to benefit from its standard notations : no real cohomological tool will be used. 

\begin{definition}[cohomology of a graph]\label{définition:cohomologie}
	Let a graph $\cal{G} = (\cal{S},\cal{A})$ and an abelian group $(G,\star)$ (which we shall take among $\R^*$, $\Z/2\Z$ and $\R$). Let us define a cochain complex by $\rm{C}^0(\cal{G},G) = G^\cal{S}$, $\rm{C}^1(\cal{G},G) = G^\cal{A}$, $\rm{C}^*(\cal{G},G) = 0$ otherwise, the differential $d : \rm{C}^0(\cal{G},G) \to \rm{C}^1(\cal{G},G)$ being given, for a function $f \in G^{\cal{S}}$ and an edge $a$ with endpoints $\xi$ and $\eta$, by $\d f(a) = f(\xi) \star f(\eta)$. Let us denote $\cohom^k(\cal{G},G)$ its cohomology. 
\end{definition}
Any group morphism $\phi : G_1 \to G_2$ induces a morphism of complexes $\rm{C}^*(\cal{G},G_1) \to \rm{C}^*(\cal{G},G_2)$ and hence a morphism in cohomology, both denoted $\phi_*$. Likewise, a graph morphism $\psi$ determines morphisms $\psi^*$ between the associated cochain complexes and the cohomologies of the graphs.

Let now $\Lambda$ be a subset of $\P(\R^{p,q})$. Given a lift $s$ of $\Lambda$, we define a cocycle $\rm{Gram}(s) \in \rm{C}^1(\cal{G}(\Lambda),\R^*)$ by associating to an edge $\{\xi,\eta\}$ of the graph $\cal{G}(\Lambda)$ the real number $-\bf{b}(s(\xi),s(\eta)) \in \R^*$. The following lemma states that the cohomology class of $\rm{Gram}(s)$ depends only on $\Lambda$ : we shall denote it $c(\Lambda) \in \cohom^1(\cal{G}(\Lambda),\R^*)$. 
\begin{lemme}
	The set of $\rm{Gram}(s)$, for $s$ a lift of $\Lambda$, forms a cohomology class in $\cohom^1(\cal{G}(\Lambda),\R^*)$. 
\end{lemme}
\begin{proof}
	Choosing another lift $s'$ of $\Lambda$ is equivalent to choosing the function $g : \Lambda \to \R^*$ such that $s' = g s$, which is a cocycle $g \in \rm{C}^0(\cal{G}(\Lambda),\R^*)$. We then have, given some edge $a$ with endpoints $\xi$ and $\eta$ :
	\[\rm{Gram}(s')(a) = -\bf{b}(s'(\xi),s'(\eta)) = -\bf{b}(s(\xi),s(\eta)) g(\xi) g(\eta). \]
	With the notations of Definition \ref{définition:cohomologie}, this can be written as $\rm{Gram}(s')(a) = \rm{Gram}(s)(a) \star \d g$, which concludes. 
\end{proof}
The following lemma states that a pair $(\Lambda,s)$ with $\Lambda \subset \P(\R^{p,q})$ such that $\Vect{\Lambda}$ is non-degenerate and $s$ a lift of $\Lambda$, considered modulo the action of $\PO(p,q)$, is determined by $\rm{Gram(s)}$.
\begin{lemme}\label{lemme:matrice_Gram_isométrie}
Let $\iota$ be a map from a subset $E$ of $\R^{p,q}$ into $\R^{p,q}$, such that $\Vect{E}$ and $\Vect{\iota(E)}$ are non-degenerate. The map $\iota$ extends to an element of $\O(p,q)$ if and only if one has $\bf{b}(x,y) = \bf{b}(\iota(x),\iota(y))$ for all $x,y \in E$. 
\end{lemme}
\begin{proof}
Suppose that $\bf{b}(x,y) = \bf{b}(\iota(x),\iota(y))$ for all $x,y \in E$. Let $(x_i)$ be a basis of $\Vect{E}$ extracted from $E$. The family $(\iota(x_i))$ has the same Gram matrix as $(x_i)$, therefore is linearly independent and forms a basis of $\Vect{\iota(E)}$. The linear map $g : \Vect{E} \to \Vect{\iota(E)}$ defined by $g(x_i) = \iota(x_i)$ is then an isometry and satisfies $g(x) = \iota(x)$ for all $x \in E$. As $\Vect{E}$ and $\Vect{\iota(E)}$ are non-degenerate, each admits an orthogonal complement : this allows extending $g$ to an element of $\O(p,q)$. The converse is immediate.
\end{proof}

We now show that a subset $\Lambda$ of $\P(\R^{p,q})$, considered modulo the action of $\PO(p,q)$, is determined by the cohomology class $c(\Lambda) \in \cohom^1(\cal{G}(\Lambda),\R^*)$.
\begin{proposition}\label{proposition:classification_bords}
Let $\Lambda_1$ and $\Lambda_2$ be two subsets of $\P(\R^{p,q})$ spanning non-degenerate subspaces of $\R^{p,q}$. Let $\iota  : \Lambda_1 \to \Lambda_2$ be a map. The following conditions are equivalent :
\begin{enumerate}
	\item $\iota$ extends to an element of $\PO(p,q)$ ;
	\item $\iota$ defines isomorphism of graphs $\cal{G}(\Lambda_1) \to \cal{G}(\Lambda_2)$ and we have $c(\Lambda_1) = \iota^* c(\Lambda_2)$ in $\cohom^1(\cal{G}(\Lambda_1),\R^*)$.
\end{enumerate}
\end{proposition}
\begin{proof}
Assume that the first condition holds. Then $\iota$ induces an isomorphism $\cal{G}(\Lambda_1) \to \cal{G}(\Lambda_2)$. Moreover, given a lift $s_1$ of $\Lambda_1$ and an isometry $\phi \in \O(p,q)$ such that $[\phi] \in \PO(p,q)$ extends $\iota$, $s_2 = \phi \circ s_1$ defines a lift of $\Lambda_2$. We have $\rm{Gram}(s_1) = \iota^* \rm{Gram}(s_2) \in \rm{C}^1(\cal{G}(\Lambda_1),\R^*)$, hence $c(\Lambda_1) = \iota^* c(\Lambda_2)$.

Conversely, assume that the second condition holds. Given two lifts $s_1,s_2$ of $\Lambda_1,\Lambda_2$, we have in $\cohom^1(\cal{G}(\Lambda_1),\R^*)$ the equality $[\rm{Gram}(s_1)] = c_{\Lambda_1} = \iota^* c_{\Lambda_2} = [\iota^* \rm{Gram}(s_2)]$, hence there exists $g \in \rm{C}^0(\cal{G}(\Lambda_1),\R^*)$ such that $\rm{Gram}(s_1) \star \d g = \iota^* \rm{Gram}(s_2)$ ($\star$ being the group law of $\R^*$). We define a second lift of $\Lambda_1$ by $s_1'(\xi) = g(\xi) s_1(\xi)$, which satisfies $\rm{Gram}(s_1') = \iota^* \rm{Gram}(s_2) = \rm{Gram}(s_2 \circ \iota)$. By Lemma \ref{lemme:matrice_Gram_isométrie}, there exists an isometry $\phi \in \O(p,q)$ such that $\phi \circ s_1' = s_2 \circ \iota$. Thus $[\phi] \in \PO(p,q)$ extends $\iota$.
\end{proof}

\textbf{Decomposition in cohomology and \positif lifts.}
Define the quotient maps $\rm{sgn}  : \R^* \to \R^*/\R_{>0} = \Z/2\Z$ and $\rm{abs}  : \R^* \to \R^*/{\pm 1} = \R_{>0}$. The canonical decomposition $\R^* = \Z/2\Z \times \R_{>0}$ determines an isomorphism $\cohom^k(\cal{G},\R^*) = \cohom^*(\cal{G},\Z/2\Z) \times \cohom^k(\cal{G},\R_{>0})$ for every $k$, given by $\rm{sgn}_* \times \rm{abs}_*$. A class $c \in \cohom^k(\cal{G},\R^*)$ satisfies $\rm{sgn}_* c = 0$ if and only if it lies in the image of $\exp_*  : \cohom^k(\cal{G},\R) \to \cohom^k(\cal{G},\R^*)$. For any subset $\Lambda$ of $\P(\R^{p,q})$, we denote by $c_\pm(\Lambda) = \rm{sgn}_* c(\Lambda)$ and $c_\R = \rm{abs}_* c(\Lambda)$. In the above decomposition, one therefore has $c(\Lambda) = (c_\pm(\Lambda),c_\R(\Lambda))$.

\begin{proposition}\label{proposition:existence_relèvement_positif}
A subset $\Lambda$ of $\P(\R^{p,q})$ admits a \positif lift if and only if the cohomology class $c_\pm(\Lambda) \in \cohom^1(\cal{G}(\Lambda),\Z/2\Z)$ vanishes.
\end{proposition}
\begin{proof}
The set of the $\rm{Gram}(s)$, with $s$ a lift of $\Lambda$, forms the cohomology class $c(\Lambda)$. Moreover, a lift $s$ is \positif if and only if $\rm{Gram}(s)$ takes values in $\R_{>0}$, that is, if $\rm{sgn}_* \rm{Gram}(s) \in \rm{C}^1(\cal{G}(\Lambda),\Z/2\Z)$ is trivial. Thus, there exists a \positif lift if and only if the cohomology class $\rm{sgn}_* c(\Lambda)$ contains the trivial cocycle, that is, if $c_\pm(\Lambda) = \rm{sgn}_* c(\Lambda) = 0$.
\end{proof}

\subsection{Topology and signature}

\label{définition:topologie_H1}
We now define a topology on the cohomology groups : we wish to define a notion of convergence corresponding to the convergence between simplices. For this, given a set $\cal{S}$, we define as follows a space $\cohom^1(\cal{S},\R^*)$ containing all the $\cohom^1(\cal{G},\R^*)$ for $\cal{G}$ a graph whose set of vertices is $\cal{S}$. We refer to Section \ref{subsection:exemples} for examples of graphs associated to simplices. 

First, define $\cal{A}$ to be the set of subsets of $\cal{S}$ with $1$ or $2$ elements and endow $\Sym(\cal{S},\R) = \R^\cal{A}$ and $(\R^*)^\cal{S}$ with the product topology. Given a graph $\cal{G}$ with vertices $\cal{S}$ and a cochain $f \in \rm{C}^1(\cal{G},\R^*)$, define $\Mat(f) \in \Sym(\cal{S},\R)$ by associating to a vertex $\{i,j\}$ the real $-f(\{i,j\})$ if it is defined and $0$ otherwise. We endow $\Sym(\cal{S},\R)$ with an action of $(\R^*)^\cal{S}$ by $[g \cdot f](s,s') = g(s) g(s') f({s,s'})$. The set of $0$-cochains $\rm{C}^0(\cal{G},\R^*) = (\R^*)^\cal{S}$ is independent from $\cal{G}$. Moreover, the map $f \mapsto \Mat(f)$ defines a bijection
\[\bigcup_{\cal{G}} \rm{C}^1(\cal{G},\R^*) \stackrel{\Mat}{\longrightarrow} \Sym(\cal{S},\R), \]
where the union is taken over all graphs $\cal{G}$ with vertex set $\cal{S}$. For such a graph $\cal{G}$, the action of $\rm{C}^0(\cal{G},\R^*)$ on $\rm{C}^1(\cal{G},\R^*)$ is thus identified with the action of $(\R^*)^\cal{S}$ on a subset of $\Sym_n(\R)$. 
\begin{definition*}
We define a topological space (with the quotient topology) by
\[\cohom^1(\cal{S},\R^*) = \Sym(\cal{S},\R)/(\R^*)^\cal{S}. \]
For any graph $\cal{G} = (\cal{S},\cal{A})$, we identify $\cohom^1(\cal{G},\R^*)$ with its image under the embedding $\cohom^1(\cal{G},\R^*) \to \cohom^1(\cal{S},\R^*)$. 
\end{definition*}
The induced topology on $\cohom^1(\cal{G},\R^*)$ is indeed the quotient topology on $\rm{C}^1(\cal{G},\R^*)/\rm{C}^0(\cal{G},\R^*)$, these spaces being identified with $(\R^*)^\cal{S}$ and $(\R^*)^\cal{A}$ endowed with the product topology. We shall for instance say, given graphs $\cal{G}, \cal{G}'$ with vertex set $\cal{S}$, that a sequence $(c_k)$ of points in $\cohom^1(\cal{G},\R^*)$ converges to $c \in \cohom^1(\cal{G}',\R^*)$ if this is the case for their images in $\cohom^1(\cal{S},\R^*)$. However, the orbits of $\Sym(\cal{S},\R)$ under the action of $(\R^*)^\cal{S}$ are not closed : thus $\cohom^1(\cal{S},\R^*)$ is not Hausdorff.

\textbf{Signature.} 
Given a simplex of $\H^{p,q}$, we wish to recover the signature $(p,q+1)$ only from its cohomology class. This leads to the following definition. Let $\cal{G}$ be a finite graph with vertex set $\cal{S}_n = \{1,\dots,n\}$. Given a cochain $f \in \rm{C}^1(\cal{G},\R^*)$, $\Mat(f)$ is naturally identified with a real symmetric matrix with coefficients $-f(\{i,j\})$ if $\{i,j\}$ is an edge, $0$ otherwise. We shall say that $f$ has signature $(p,q)$ (denoted $\rm{sgn}(f) = (p,q)$), rank $r=p+q$ or is invertible if this holds for $\Mat(f)$. 
\begin{lemme}
The signature, thus the rank are constant on each cohomology class. 
\end{lemme}
\begin{proof}
Given $g \in \rm{C}^0(\cal{G},\R^*)$ and $f \in \rm{C}^1(\cal{G},\R^*)$, denoting $D$ the diagonal matrix $\diag(g(1),\dots,g(n))$, we have $\Mat(f + \d g) = D\, \Mat(f)\, \transp{D}$. Sylvester's law of inertia states that the signature of a symmetric matrix is invariant under congruence : we deduce that $f$ and $f + \d g$ have the same signature. 
\end{proof}
\begin{definition}
We define the signature of a cohomology class in $\cohom^1(\cal{G},\R^*)$ as that of its representatives. We define similarly its rank and its invertibility. The signature defines a partition of $\cohom^1(\cal{G},\R^*)$, thus of the subspace $\cohom^1(\cal{G},\R) \stackrel{\exp_*}{\hookrightarrow} \cohom^1(\cal{G},\R^*)$. 
\end{definition}

\subsection{Study of simplices via graphs}

We have classified the isometry classes of subsets $\Lambda$ of $\P(\R^{p,q})$ by cohomology classes $c(\Lambda) \in \cohom^1(\cal{G}(\Lambda),\R^*)$ (Proposition \ref{proposition:classification_bords}), the subsets admitting a \positif lift being classified by real cohomology classes (Proposition \ref{proposition:existence_relèvement_positif}). From this, one deduces a classification of simplices in $\H^{p,q}$ (Proposition \ref{proposition:classification_simplexes}). We then seek to do the converse : given a graph, we try to find a subset of $\P(\R^{p,q})$ that defines it. This yields constraints on the possible signatures $(p,q)$ (Lemma \ref{lemme:cohomologie_vers_permutation} and Corollary \ref{corollaire:permutation_vers_cohomologie}). This allows reducing questions about simplices to graph-theoretic questions (Proposition \ref{proposition:classification_cycles} and Corollary \ref{corollaire:simplexe_H22_fini}), in the spirit of the \hyperref[thm_principal]{main theorem}.

\textbf{Classification of marked simplices.}
Recall that a \hyperref[définition:polytope]{marked simplex} is a bijection $m$ from $\cal{S}_n = \{1,\dots,n\}$ onto a simplex $\Lambda$. This bijection defines an isomorphism $m$ from a graph $\cal{G}(m)$ with vertex set $\cal{S}_n$ to $\cal{G}(\Lambda)$. We then define $c_\R(m) = m^* c_\R(\Lambda)$, which lies in $\cohom^1(\cal{G}(m),\R)$. \hyperref[définition:topologie_H1]{Recall} that $\cohom^1(\cal{G}(m),\R)$ embeds into $\cohom^1(\cal{S}_n,\R)$. We thus deduce the following classification of marked simplices in $\H^{p,q}$.

\begin{proposition}\label{proposition:classification_simplexes}
Let $\cal{S} = \{1,\dots,p+q+1\}$. The map $c_\R$ defines a bijection from the set of isometry classes of marked simplices in $\H^{p,q}$ to the set of cohomology classes $c \in \cohom^1(\cal{S},\R)$ of signature $\rm{sgn}(c) = (p,q+1)$.
\end{proposition}

\begin{proof}
Injectivity is a consequence of Propositions \ref{proposition:classification_bords} and \ref{proposition:existence_relèvement_positif}. To prove surjectivity, consider $c \in \cohom^1(\cal{S},\R)$, which therefore lies in $\cohom^1(\cal{G},\R)$ for some graph $\cal{G}$ with vertex set $\cal{S}$. Choose a representative $f \in \rm{C}^1(\cal{G},\R)$ of $c$. The matrix $\Mat(\exp_* f)$ defines a non-degenerate bilinear form $\bf{b}$ on $\R^{p+q+1}$, identifying it with $\R^{p,q+1}$. The canonical basis $(e_1,\dots,e_{p+q+1})$ defines a map $m  : k \in \cal{S} \mapsto \P \{e_k\}$, whose image we denote by $\Lambda = \P\{e_1,\dots,e_{p+q+1}\}$. Since the entries of $\Mat(\exp_* f)$ lie in $\R_{\leq 0}$, we have $\bf{b}(x,x) \leq 0$ for $x \in \Conv(e_1,\dots,e_{p+q+1})$, so the map $\P \{e_k\} \in \Lambda \mapsto e_k$ defines a \positif lift of $\Lambda$ : hence $m$ is a marked simplex. By definition, it satisfies $\cal{G}(m) = \cal{G}$ and $m^* c_\R(\Lambda) = c$, which concludes the proof.
\end{proof}

\textbf{Dimension of cohomologies.}
Let $\cal{S}_n  =\{1,\dots,n\}$. By \hyperref[définition:topologie_H1]{definition}, $\cohom^1(\cal{S}_n,\R)$ minus the point corresponding to $0 \in \Sym(n,\R)$ is the image of $\P \Sym(n,\R)$ by a continuous map : therefore, $\cohom^1(\cal{S}_n,\R)$ is compact but not Hausdorff. It is partitioned into a finite number of cells: the $\cohom^1(\cal{G},\R)$ with $\cal{G}$ of vertices $\cal{S}_n$, identified with finite-dimensional vector spaces. 

\begin{lemme}\label{lemme:dimension_H1}
Let $\cal{G} = (\cal{S},\cal{A})$ be a finite graph. We have $\dim \cohom^1(\cal{G},\R) = \Card \cal{A} - \Card \cal{S} + \dim \cohom^0(\cal{G},\R)$ and $\dim \cohom^0(\cal{G},\R)$ is the number of bipartite connected components of $\cal{G}$. Similarly, we have $\dim \cohom^1(\cal{G},\Z/2\Z) = \Card \cal{A} - \Card \cal{S} + \dim \cohom^0(\cal{G},\Z/2\Z)$ and $\dim \cohom^0(\cal{G},\Z/2\Z)$ is the number of connected components of $\cal{G}$. 
\end{lemme}
\begin{proof}
The chain complex defining $\cohom^*(\cal{G},\R)$ is given by $\rm{C}^0 = \R^{\cal{S}}$, $\rm{C}^1 = \R^{\cal{A}}$, and $\rm{C}^* = 0$ otherwise. From this, we deduce $\dim \cohom^0(\cal{G},\R) - \dim \cohom^1(\cal{G},\R) = \Card \cal{S} - \Card \cal{A}$, yielding the desired equality. To prove the assertion on $\cohom^0(\cal{G},\R)$, one may assume that $\cal{G}$ is connected and then sum over the connected components. Under this assumption, if $g  : \cal{S} \to \R$ is non-zero and satisfies $\d g = 0$, then the sets $\cal{S}_+ = \{s \in \cal{S} \mid g(s) > 0\}$ and $\cal{S}_- = \{s \in \cal{S} \mid g(s) < 0\}$ define a partition of $\cal{S}$. Every $0$-cocycle is constant on each of these sets, hence a multiple of $g$ : this proves $\dim \cohom^0(\cal{G},\R) = 1$. Conversely, if $\cal{S} = \cal{S}_+ \sqcup \cal{S}_-$ is a partition into stable sets, the function $g  : \cal{S} \to \R$ taking values $\pm 1$ on $\cal{S}_\pm$ is a cocycle, and every cocycle is a multiple of $g$, which concludes the proof for the cohomology with coefficients in $\R$. The case of cohomology with coefficients in $\Z/2\Z$ is similar. 
\end{proof}

\begin{remarque}\label{remarque:nombre_enveloppes_convexes}
If a finite subset $\Lambda$ of $P(\R^{p,q})$ admits a \positif lift, then the proof of Lemma \ref{lemme:relèvement_unique} shows it has exactly $2^{\dim \cohom^0(\cal{G},\R) - 1}$ distinct convex hulls, where the $-1$ comes from projectivisation. 
\end{remarque}

\textbf{Possible signatures on a graph.}
\begin{lemme}\label{lemme:cohomologie_vers_permutation}
Let $\cal{G}$ be a graph with vertex set $\{1,\dots,n\}$. Suppose there exists an invertible cohomology class $c \in \cohom^1(\cal{G},\R^*)$. Then there exists a permutation $\sigma \in \frak{S}_n$ such that for every $k$, the vertices $k$ and $\sigma(k)$ are adjacent in $\cal{G}$. Moreover, if there exists $c \in \cohom^1(\cal{G},\R)$ invertible of signature $(p,q)$, then one can choose $\sigma$ of signature $(-1)^p$.
\end{lemme}

\begin{proof}
Let $(m_{i,j}) = \Mat(f)$ and $\epsilon(\sigma)$ denote the signature of a permutation $\sigma$. Let $f$ be a representative of $c$. The determinant $\det \Mat(f)$ is non-zero and can be written as
\[\sum_{\sigma \in \frak{S}_n} \epsilon(\sigma) \prod_{1 \leq k \leq n} m_{k,\sigma(k)}. \]
In particular, there exists a permutation $\sigma$ for which the corresponding term $\prod_{1 \leq k \leq n} m_{k,\sigma(k)}$ is non-zero, meaning that $k$ and $\sigma(k)$ are adjacent for all $k$. Moreover, if $c \in \cohom^1(\cal{G},\R)$ is invertible of signature $(p,q)$, then $\Mat(f)$ is a symmetric matrix of signature $(p,q)$, so $\det \Mat(f)$ has the sign of $(-1)^q$. Therefore, there exists $\sigma$ such that $\epsilon(\sigma) \prod_{1 \leq k \leq n} m_{k,\sigma(k)}$ has the sign of $(-1)^q$. Since the $m_{k,\sigma(k)}$ lie in $\R_{\leq 0}$, this product has the same sign as $\epsilon(\sigma) (-1)^n$, hence $\epsilon(\sigma) = (-1)^{n+q} = (-1)^p$.
\end{proof}

To prove a converse (Corollary \ref{corollaire:permutation_vers_cohomologie}), we shall use the following fact. Given an integer $n > 0$ and real numbers $a_1,\dots,a_n < 0$, the following matrix is symmetric, hence defines a quadratic form :
\[\rm{circ}(a_1,\dots,a_n) = 
\begin{pmatrix}
0 & a_1 & 0 & \dots & a_n \\
a_1 & 0 & a_2 & \dots & 0 \\
0 & a_2 & 0 & \dots & 0 \\
\dots & \dots & \dots & \dots & a_{n-1} \\
a_n & 0 & 0 & a_{n-1} & 0
\end{pmatrix}. \]
\begin{fait}\label{fait:signature_cycle}
If $n \equiv 0 [4]$ and $\prod_{k\, \text{even}} a_k = \prod_{k\, \text{odd}} a_k$, then $\rm{circ}(a_1,\dots,a_n)$ is degenerate, with signature $(\frac{n}{2}-1,\frac{n}{2}-1)$. Otherwise, it is non-degenerate with signature $(\frac{n}{2},\frac{n}{2})$ if $n$ is even, $(\frac{n-1}{2},\frac{n+1}{2})$ if $n \equiv 1 [4]$, and $(\frac{n+1}{2},\frac{n-1}{2})$ if $n \equiv 3 [4]$. 
\end{fait}

For any integer $n \geq 2$, let $C_n$ denote the cycle of length $n$ : it is the graph whose set of vertices is $\Z/n\Z$, the edges being the $\{k,k+1\}$ for $k$ in $\Z/n\Z$. Define $C_1$ as the graph with one vertex and a loop. 
\begin{proposition}\label{proposition:classification_cycles}
The cohomology of $C_n$ is as follows. 
\begin{enumerate}
	\item If $n \equiv 0 [4]$, then $\cohom^1(C_n,\R)$ is one-dimensional and each non-zero class has signature $(\frac{n}{2},\frac{n}{2})$ and the zero class has signature $(\frac{n}{2}-1,\frac{n}{2}-1)$. 
	\item If $n \equiv 1 [4]$, then $\cohom^1(C_n,\R)$ consists of a single point, of signature $(\frac{n-1}{2},\frac{n+1}{2})$. 
	\item If $n \equiv 2 [4]$ and $n \neq 2$, then $\cohom^1(C_n,\R)$ is one-dimensional and each class has signature $(\frac{n}{2},\frac{n}{2})$. $\cohom^1(C_2,\R)$ consists of a single point, of signature $(1,1)$. 
	\item If $n \equiv 3 [4]$, then $\cohom^1(C_n,\R)$ consists of a single point, of signature $(\frac{n+1}{2},\frac{n-1}{2})$. 
\end{enumerate}
For example, if $n$ is odd, there exists a unique graph simplex of $C_n$ : it is realised with a single signature and is unique up to isometry. 
\end{proposition}
\begin{proof}
This is a consequence of the classification of simplices (Proposition \ref{proposition:classification_simplexes}), the computation of the dimensions of the $\cohom^1$ (Lemma \ref{lemme:dimension_H1}), and their partition according to signature (obvious for $n=1$ and otherwise given by fact \ref{fait:signature_cycle}). 
\end{proof}

\begin{corollaire}\label{corollaire:permutation_vers_cohomologie}
Let $\cal{G}$ be a finite graph with vertex set $\{1,\dots,n\}$ and let $\sigma \in \frak{S}_n$ be a permutation such that $s$ and $\sigma(s)$ are adjacent for all $s \in \{1,\dots,n\}$. Let $l_k(\sigma)$ denote the number of cycles of length $k$ in the cycle decomposition of $\sigma$, and define integers $p,q$ by
\[2p = n - \sum_{k \equiv 1 [4]} l_k + \sum_{k \equiv 3 [4]} l_k,\ q = n - p. \]
There exists $c \in \cohom^1(\cal{G},\R)$ invertible of signature $(p,q)$. 
\end{corollaire}
\begin{proof}
Define a graph $\cal{G}_\sigma$ with vertex set $\{1,\dots,n\}$ and edges $\{k,\sigma(k)\}$ for $k \in \{1,\dots,n\}$. Let $\cal{G}_i$ ($i \in I$) denote the connected components of $\cal{G}_\sigma$ : these are cycles. We have
\[\cohom^1(\cal{G}_\sigma,\R^*) = \prod_i \cohom^1(\cal{G}_i,\R^*) \]
and Proposition \ref{proposition:classification_cycles} gives, for each $i$, a cohomology class $c_i \in \cohom^1(\cal{G}_i,\R^*)$ invertible with explicit signature. The product of these classes is thus a cohomology class $c \in \cohom^1(\cal{G}_\sigma,\R^*)$, whose signature is the sum of the signatures of the $c_i$ (given by Proposition \ref{proposition:classification_cycles}), equal to $(p,q)$. Since $\cal{G}_\sigma$ is a subgraph of $\cal{G}$, there exists (by \hyperref[définition:topologie_H1]{definition}) a sequence $(c_k)$ of elements of $\cohom^1(\cal{G},\R)$ converging to $c$ : intuitively, the $c_k$ correspond to simplices whose graph is $\cal{G}$, on which we let degenerate the edges not corresponding to some $(k,\sigma(k))$. For sufficiently large $k$, the signature of $c_k$ satisfies $\rm{sgn}(c_k) = (p_k,q_k) \geq (p,q)$. Moreover, since $n = p+q \geq p_k + q_k$, it follows that $\rm{sgn}(c_k) = (p,q)$, which completes the proof. 
\end{proof}

\section{Finiteness criterion on the volume}
\label{section:critère}

This section is devoted to the proof of the \textbf{\hyperref[thm_principal]{main theorem}}. We refer again to Section \ref{subsection:exemples} for concrete examples. Recall that, given a graph and a set $I$ of vertices of this graph, we denote by $\partial I$ the set of vertices which are not in $I$ but adjacent to a vertex of $I$. $I$ is said to be stable if there is no edge between vertices of $I$. 

\noindent \textbf{Main theorem.} \label{thm_principal}
\textit{Let $\Lambda$ be a simplex of $\H^{p,q}$. $\Lambda$ has infinite volume if, and only if, the graph $\cal{G}(\Lambda)$ contains a non-empty, stable set of vertices $I$ with cardinality at least $\Card \partial I$. }

\subsection{Proof}

Let us note $n = p+q+1$. The proof proceeds as follows. Lemmas \ref{lemme:volume_simplexe} and \ref{lemme:simplification_Delta} reduce the study of the volume of $\Lambda$ to that of a subset $\Delta$ of $\R_{>0}^n$. A further change of coordinates mallows writing the volume of $\Delta$ as the integral, over a convex cone $P$ of $\R^n$, of $e^\sigma$ where $\sigma$ is a linear form. From the fact that $P$ is the intersection of the half-spaces $\{(p_k) \mid p_i + p_j \leq 0\}$ with $\{i,j\}$ an edge of $\cal{G}(\Lambda)$, we infer (Lemma \ref{lemme:divergence_poids}) that this integral diverges if, and only if, there exists a non-zero $p \in P$ satisfying $\sigma(p) \leq 0$. Finally, the existence of such a $p$ is related to the graph $\cal{G}(\Lambda)$ (Lemma \ref{lemme:poids_graphe}).

To begin with, choose a \positif lift $s$ of $\Lambda$ whose image we denote by $\{x_1,\dots,x_n\}$, and define a linear map $\alpha  : \R^n \stackrel{\sim}{\longrightarrow} \R^{p,q+1}$ by $\alpha(t_1,\dots,t_n) = \sum t_k x_k$. In $\R^n$ and $\R^{p,q+1}$, identify each tangent space with $\R^n$ or $\R^{p,q+1}$ respectively. The volume forms $\vol^{\R^n}$ and $\vol^{\R^{p,q+1}}$ are related by $\alpha^*\vol^{\R^{p,q+1}} = \det(\alpha) \vol^{\R^n}$ for some non-zero real $\det(\alpha) = \vol^{\R^{p,q+1}}(x_1,\dots,x_n)$. Define a quadratic form $f$ on $\R^n$ by
\[f(t_1,\dots,t_n) = -\bf{b}(\alpha(t),\alpha(t)) = \sum -t_i t_j \bf{b}(x_i,x_j). \]
\begin{lemme}\label{lemme:volume_simplexe}
	Endow $\Hc^{p,q}$ and $\R^n_{>0}$ with the measures coming from their respective (pseudo-Riemannian) metrics. We have
	\[\Vol^{\Hc^{p,q}}(\Conv(s)) = n \abs{\det(\alpha)} \Vol^{\R^n_{>0}} \left\lbrace t \in \R^n_{>0} \mid f(t) \leq 1 \right\rbrace. \]
\end{lemme}
\begin{proof}
Notice that $f > 0$ on $\R^n_{>0}$. We define an embedding of $\S^{n-1}_{>0} = \S^{n-1} \cap \R_{>0}^n$ into $\Hc^{p,q}$, whose image is a lift of $\Conv(s)$:
\[\begin{array}{rrcl}
	\psi  : & \S^{n-1}_{>0} & \longrightarrow & \Hc^{p,q} \\
	\ & t = (t_1,\dots,t_n) & \longmapsto & f(t)^{-\frac12} \alpha(t). 
\end{array} \]
Let $s \in \S^{n-1}_{>0}$. Let $\sf{n}$ denote the unit normal vector at $s$ to $\S^{n-1}_{>0}$ (viewed as a submanifold of $\R^n$) : $\sf{n}$ is naturally identified with $s$. Then $f(t)^{-\frac12} \d \alpha (\sf{n})$ is a normal vector to $\Hc^{p,q}$ in $\R^{p,q+1}$, naturally identified with $\psi(s)$. By the definition of $\psi$, we have
\begin{equation}\label{éq:différentielle_psi}
	\d \psi(\sf{v}) = f(t)^{-\frac12}\ \d \alpha(\sf{v}) + \d(f^{-\frac12})(\sf{v})\ \alpha(s). 
\end{equation}
The volume forms of $\Hc^{p,q}$ and $\S^{n-1}$ are given by $\vol^{\Hc^{p,q}}_x = \vol^{\R^{p,q+1}}(\cdot,\dots,\cdot,x)$ (see Section \ref{section:préliminaires}) and $\vol^{\S^{n-1}}_s = \vol^{\R^n}(\cdot,\dots,\cdot,s)$. Let $\sf{v}_1,\dots,\sf{v}_{n-1} \in \Tan_s \S^{n-1}_{>0}$. We have
\begin{align*}
	&(\psi^* \vol^{\Hc^{p,q}}) (\sf{v}_1,\dots,\sf{v}_{n-1}) \\
	&= \vol^{\Hc^{p,q}}(\d\psi (\sf{v}_1),\dots,\d\psi (\sf{v}_{n-1})) \\
	&= \vol^{\R^{p,q+1}} (\d\psi (\sf{v}_1),\dots,\d\psi (\sf{v}_{n-1}),\psi(s)) \tag*{so by \eqref{éq:différentielle_psi},} \\
	&= \vol^{\R^{p,q+1}} \left( f(s)^{-\frac12}\ \d\alpha(\sf{v}_1),\dots,f(s)^{-\frac12}\ \d\alpha(\sf{v}_{n-1}),f(s)^{-\frac12}\ \d\alpha(\sf{n}) \right) \\
	&= f(s)^{-\frac{n}{2}}\ (\alpha^*\vol^{\R^{p,q+1}}) (\sf{v}_1,\dots,\sf{v}_{n-1},\sf{n}) \\
	&= f(s)^{-\frac{n}{2}}\ \det(\alpha)\ \vol^{\R^n}(\sf{v}_1,\dots,\sf{v}_{n-1},\sf{n}) \\
	&= f(s)^{-\frac{n}{2}}\ \det(\alpha)\ \vol^{\S^{n-1}}(\sf{v}_1,\dots,\sf{v}_{n-1}). 
\end{align*}
This shows that 
\[\psi^* \vol^{\Hc^{p,q}} = f(s)^{-\frac{n}{2}}\ \det(\alpha)\ \vol^{\S^{n-1}}. \]
We are reduced to studying the integrability of $\psi^* \vol^{\Hc^{p,q}}$ over $\S^{n-1}_{>0}$. This is equivalent to the convergence of
\begin{align*}
	& \int_{\S^{n-1}_{>0}} f(s)^{-\frac{n}{2}}\ \d\!\vol^{\S^{n-1}}(s) \\
	=& \Vol^{\S^{n-1} \times \R_{>0}} \left\lbrace (s,h) \in \S^{n-1}_{>0} \times \R_{>0} \mid f(s)^{-\frac{n}{2}} \geq h \right\rbrace \tag*{by Fubini} \\
	=& \int_{\R_{>0}} \Vol^{\S^{n-1}} \left\lbrace s \in \S^{n-1}_{>0} \mid f(s)^{-\frac{n}{2}} \geq h \right\rbrace \d h \tag*{by Fubini} \\
	=& \int_{\R_{>0}} \Vol^{\S^{n-1}} \left\lbrace s \in \S^{n-1}_{>0} \mid f(s) \leq h^{-\frac{2}{n}} \right\rbrace \d h \\
	=& \int_{\R_{>0}} n\ r^{n-1} \Vol^{\S^{n-1}} \left\lbrace s \in \S^{n-1}_{>0} \mid f(s) \leq r^{-2} \right\rbrace \d r \tag*{by the change of variable $r = h^{\frac{1}{n}}$} \\
	=& n \int_{\R_{>0}} r^{n-1} \Vol^{\S^{n-1}} \left\lbrace s \in \S^{n-1}_{>0} \mid f(rs) \leq 1 \right\rbrace \d r \tag*{since $f$ is homogeneous} \\
	=& n \int_{\R_{>0}} r^{n-1} \Vol^{\S^{n-1}} \left\lbrace s \in r \S^{n-1}_{>0} \mid f(s) \leq 1 \right\rbrace \d r \tag*{so by Fubini,} \\
	=& n \Vol^{\R^n} \left\lbrace t \in \R^n_{>0} \mid f(t) \leq 1 \right\rbrace. 
\end{align*}
This concludes the proof of Lemma \ref{lemme:volume_simplexe}. 
\end{proof}

Recall that the set $\cal{A}$ of edges of the graph $\cal{G}(\Lambda)$ is given by $\cal{A} = \{\{i,j\} \mid \bf{b}(x_i,x_j) < 0\}$. Let $\Delta = \left\lbrace (t_1,\dots,t_n) \in \R^n_{>0} \mid \forall {i,j} \in \cal{A} \quad t_i t_j \leq 1 \right\rbrace$.
\begin{lemme}\label{lemme:simplification_Delta}
The volume of $\left\lbrace t \in \R^n_{>0} \mid f(t) \leq 1 \right\rbrace$ is finite if and only if the volume of $\Delta$ is finite.
\end{lemme}
\begin{proof}
Recall that $f(t) = \sum -t_i t_j \bf{b}(x_i,x_j)$, the sum being on the couples $(i,j)$ such that $\{i,j\}$ is an edge. Define
\[\begin{array}{rrcl}
	g  : & \R^n & \longrightarrow & \R \\
	\ & t = (t_1,\dots,t_n) & \longmapsto & \sup_{\{i,j\} \in \cal{A}} t_i t_j
\end{array} \]
and choose constants $C_1,C_2$ such that for all $(i,j) \in \cal{A}$, we have $0 < C_1 \leq -\bf{b}(x_i,x_j) \leq C_2$. We hence have, for all $t \in \R^n_{>0}$ :
\begin{equation*}
	\frac{1}{\Card(\cal{A})} \frac{C_1}{C_2}\ g(t) \leq f(t) \leq 2 \Card(\cal{A}) \frac{C_2}{C_1} g(t), 
\end{equation*}
the factor $2$ being due to the fact that an edge $\{i,j\}$ with $i \neq j$ corresponds to the two pairs $(i,j)$ and $(j,i)$. Since $\Delta = \left\lbrace t \in \R^n_{>0} \mid g(t) \leq 1 \right\rbrace$, the previous estimate shows that
\begin{equation*}
	\left( \frac{1}{\Card(\cal{A})} \frac{C_1}{C_2} \right) \cdot \Delta \subset \left\lbrace t \in \R^n_{>0} \mid f(t) \leq 1 \right\rbrace \subset \left( 2 \Card(\cal{A}) \frac{C_2}{C_1} \right) \cdot \Delta. 
\end{equation*}
Thus, the finiteness of the volume of $\Delta$ is equivalent to the finiteness of the volume of $\left\lbrace t \in \R^n_{>0} \mid f(t) \leq 1 \right\rbrace$ (with respect to the volume form on $\R^n$). This proves Lemma \ref{lemme:simplification_Delta}. 
\end{proof}

Define $\sigma  : \R^n \to \R$ by $\sigma(p_1,\dots,p_n) = \sum p_k$, as well as a convex cone $P = \lbrace (p_1,\dots,p_n) \in \R^n \mid \forall \{i,j\} \in \cal{A} \quad p_i + p_j \leq 0 \rbrace$. $P$ is diffeomorphic to $\Delta$ via the map
\[\begin{array}{rcl}
\R^n & \longrightarrow & \R^n_{>0}, \\
(p_1,\dots,p_n) & \longmapsto & (e^{p_1},\dots,e^{p_n}), 
\end{array} \]
whose Jacobian at $p \in \R^n$ is $e^{\sigma(p)}$. We thus have
\begin{equation}\label{eq:Delta_P}
\vol \Delta = \int_P e^{\sigma(p)} \d p. 
\end{equation}

\begin{lemme}\label{lemme:divergence_poids}
The integral \eqref{eq:Delta_P} diverges if and only if there exists a non-zero $p \in P$ such that $\sigma(p) \geq 0$. 
\end{lemme}

\begin{proof}
\textbf{Convergent case. }Suppose that there does not exist a non-zero $p \in P$ such that $\sigma(p) \geq 0$. For each $t \in \R$, define an affine hyperplane of $\R^n$ by $H_t = \sigma^{-1}(t)$. We can rewrite \eqref{eq:Delta_P} as follows:
\begin{align}
	\int_{P} e^{\sigma(p)}\ \d p =& \int_{\R_{<0}} e^t \vol^{H_t} (P \cap H_t)\ \d t \tag*{by Jacobi} \nonumber \\
	=& \int_{\R_{<0}} e^t t^{n-1} \vol^{H_{-1}} (P \cap H_{-1})\ \d t \tag*{since $P$ is a cone} \nonumber \\
	=& (-1)^{n-1} \Gamma(n) \vol^{H_{-1}} (P \cap H_{-1}). \label{eq:volume_P}
\end{align}
This integral converges if and only if the volume of $P \cap H_{-1}$ is finite. The assumption that there does not exist a non-zero $p \in P$ with $\sigma(p) \geq 0$ implies that the cone $P$ is a union of half-lines intersecting the affine hyperplane $H_{-1}$. Since $P$ is moreover closed, the intersection $P \cap H_{-1}$ is compact, and therefore has finite volume : hence the integral \eqref{eq:volume_P} converges. 

\textbf{Divergent case. }Suppose there exists a non-zero $p \in \R^n$ such that $\sigma(p) \geq 0$. Let $U$ be a bounded open subset of $\R^n$ contained in $P$. Since $P$ is a convex cone, the set $\lbrace u + tp \mid u \in U,\ t \geq 0 \rbrace$ is a subset of $P$ with infinite measure on which $\sigma$ is bounded below. The integral of $e^\sigma$ over this set diverges, and therefore the integral \eqref{eq:Delta_P} also diverges. 
\end{proof}

\begin{lemme}\label{lemme:poids_entiers}
Suppose that there exists a non-zero $p \in P$ such that $\sigma(p) \geq 0$. There exists $(p^*_1,\dots,p^*_n) = p^* \in P$ non-zero such that $\sigma(p^*) \geq 0$ and $p^*_k \in \{-1,0,1\}$ for all $k$. 
\end{lemme}

\begin{proof}
Given $p = (p_1,\dots,p_n) \in P$ whose coefficients satisfy $-1 \leq p_k \leq 1$, denote $I_-(p) = \{k \mid -1 < p_k < 0\}$ and $I_+(p) = \{k \mid 0 < p_k < 1\}$. If $I_-(p) \cup I_+(p)$ is not empty, for $t > 0$ sufficiently small, we modify $p$ as follows. 
\begin{itemize}
	\item If $\Card I_+(p) \geq \Card I_-(p)$, replace the $p_k$ by $p_k + t$ for $k \in I_+(p)$ (ensuring $t$ is small enough so that $p_k + t \leq 1$) and by $p_k - t$ for $k \in I_-(p)$ (ensuring $p_k - t \geq -1$). 
	\item If $\Card I_+(p) \leq \Card I_-(p)$, replace the $p_k$ by $p_k - t$ for $k \in I_+(p)$ (ensuring $p_k - t \geq 0$) and by $p_k + t$ for $k \in I_-(p)$ (ensuring $p_k + t \leq 0$). 
\end{itemize}
Moreover, by taking $t$ maximal under these constraints, one obtains a new $p' \in P$ satisfying the strict inclusion $I_-(p') \cup I_+(p') \subsetneq I_-(p) \cup I_+(p)$. Furthermore, $\sigma(p') \geq \sigma(p)$. 

Suppose then that $p  = (p_1,\dots,p_n) \in P$ is non-zero such that $\sigma(p) \geq 0$. Up to acting by $\R_{>0}$, we may assume $-1 \leq p_k \leq 1$ for all $k$. Applying the previous modification to $p$ as many times as possible, we obtain $p^*\in P$ such that $I_+(p^*) = I_-(p^*) = \varnothing$ (that is $p^*_k \in \{-1,0,1\}$) and $\sigma(p^*) \geq 0$, as desired. 
\end{proof}

Combining the previous lemmas, we obtain the following condition : $\Lambda$ has finite volume if and only if there exists a non-zero $p \in P$ such that $\sigma(p) \geq 0$, which can then be taken with coordinates in $\{-1,0,1\}$. Let us interpret this condition in terms of the graph $\cal{G}(\Lambda)$. 

\begin{lemme}\label{lemme:poids_graphe}
There exists a non-zero $p \in P$ such that $\sigma(p) \geq 0$ if and only if there exists a non-empty set $I$ of the vertex set $\Lambda$ of the graph $\cal{G}(\Lambda)$, which is stable and such that $\Card I \geq \Card \partial I$. 
\end{lemme}

\begin{proof}
Suppose that $p = (p_1,\dots,p_n) \in P$ is non-zero such that $\sigma(p) \geq 0$. By Lemma \ref{lemme:poids_entiers}, we may further assume that $p_i \in \{-1,0,1\}$. We define a non-empty stable set by $I = \{i \mid p_i = 1\}$, so that $p_i = -1$ for any $i \in \partial I$. We thus have
\[0 \leq \sigma(p) \leq \Card I - \Card \partial I. \]
Conversely, if a vertex set $I$ satisfies the condition in the statement, define $(p_1,\dots,p_n)$ by setting $p_i = 1$ for $i \in I$, $p_i = -1$ for $i \in \partial I$, and $p_i = 0$ otherwise. Then $p \in P$ by construction, and since $\Card I \geq \Card \partial I$, we also have $\sigma(p) \geq 0$. 
\end{proof}

\begin{proof}[Proof of the {\hyperref[thm_principal]{main theorem}}]
We have proved:
\begin{align*}
	& \text{$\Conv(s)$ has finite volume} \\
	\Leftrightarrow\ & \text{$\left\lbrace t \in \R^n_{>0} \mid f(t) \leq 1 \right\rbrace$ has finite volume} \tag{Lemma \ref{lemme:volume_simplexe}} \\
	\Leftrightarrow\ & \text{$\Delta$ has finite volume} \tag{Lemma \ref{lemme:simplification_Delta}} \\
	\Leftrightarrow\ & \text{there exists a non-zero $p \in P$ such that $\sigma(p) \geq 0$} \tag{Lemma \ref{lemme:divergence_poids}} \\
	\Leftrightarrow\ & \text{there exists $\exists I \subset \Lambda$ non-empty, stable, such that $\Card I \geq \Card \partial I$.} \tag{Lemma \ref{lemme:poids_graphe}}
\end{align*}
This concludes the proof of the theorem. 
\end{proof}

\begin{remarque}
The \hyperref[thm_principal]{main theorem} provides a criterion independent from the signature $(p,q)$. One would similarly obtain a criterion for the finiteness of the volume of simplices in $\S^{p-1,q}$, by considering the negative lifts, obtained by replacing the condition $\bf{b}(x_i,x_j) \leq 0$ with $\bf{b}(x_i,x_j) \geq 0$. 
\end{remarque}

\subsection{Existence of simplices of infinite volume}
\label{subsection:existence_simplexes}

Given a signature $(p,q)$, we seek to determine whether in $\H^{p,q}$ there exist simplices of infinite volume or ideal simplices of infinite volume. 

\begin{proposition}\label{proposition:existence_simplexes_infinis}
Let $p,q \geq 1$  integers. There exists a non-ideal simplex of infinite volume in $\H^{p,q}$. If $p \geq 3$, then there exists an ideal simplex of infinite volume in $\H^{p,q}$. 
\end{proposition}
\begin{proof}
Let us denote by $\Lambda_0$ the unique ideal simplex of $\H^{1,0} = \R$ : its vertices are the two isotropic lines of $\R^{1,1}$. Let $\Lambda_1$ be a simplex of $\H^{p-1,q-1}$. The simplices $\Lambda_0$ and $\Lambda_1$ are sets of lines of $\R^{1,1}$ and $\R^{p-1,q}$ respectively, which we consider as subspaces of the orthogonal sum $\R^{1,1} \oplus \R^{p-1,q} = \R^{p,q+1}$. Define a set of lines of $\R^{p,q+1}$ by $\Lambda = \Lambda_0 \cup \Lambda_1$. Positive lifts $s_0 : \Lambda_0 \to \R^{p,q+1}$ and $s_1 : \Lambda_1 \to \R^{p,q+1}$ define a lift $s = s_0 \cup s_1$ of $\Lambda$. We conclude by observing that $\R$ acts freely by isometries on the factor $\R^{1,1}$ while preserving $\Conv(s)$. Suppose now $p \geq 3$ : we would like $\Lambda$ to be an ideal simplex. For this, we only need that $\Lambda_1$ be an ideal simplex. Let us choose a set of $p+q-1$ distinct points of the boundary of $\H^{p-1}$, that is to say isotropic lines in $\R^{p-1,1}$. We can approximate them by a set of $p+q-1$ isotropic lines of $\R^{p-1,q}$ generating this space : this indeed defines an ideal simplex $\Lambda_1$ of $\H^{p-1,q-1}$. 
\end{proof}

In the spaces $\H^{0,q} = \R\P^q$, $\H^{p,0} = \H^p$ and $\H^{1,q}$, the question is trivial or classical : there exist ideal simplices only in $\H^{p,0} = \H^p$ ($p \geq 1$) and the volume of simplices is bounded in each of the spaces $\H^{p,0} = \H^p$ ($p \geq 2$) and $\H^{0,q} = \R\P^q$ ($q \geq 1$). $\H^{2,1} = \AdS^3$ contains an ideal simplex of infinite volume (example \ref{étape:exemple_couronne}) and the case of $\H^{2,2}$ is surprising (see Corollary \ref{corollaire:simplexe_H22_fini}). To my knowledge, it is not known whether there exist ideal simplices of infinite volume in $\H^{2,q}$ with $q \geq 3$. 

\begin{corollaire}\label{corollaire:simplexe_H22_fini}
	Every ideal simplex of $\H^{2,2}$ has finite volume. 
\end{corollaire}
\begin{proof}
	Let $\Lambda$ be an ideal simplex of $\H^{2,2}$ : it therefore has $5$ vertices by \hyperref[définition:polytope]{definition}. By Lemma \ref{lemme:cohomologie_vers_permutation}, there exists $\sigma \in \frak{S}(\Lambda)$, of signature $1$ such that $\xi$ and $\sigma(\xi)$ are adjacent in $\cal{G}(\Lambda)$. Since $\cal{G}(\Lambda)$ has no loop (because $\Lambda$ is ideal), $\sigma$ has no fixed point. If $\sigma$ decomposed as a $3$-cycle and a transposition, it would have signature $-1$ : $\sigma$ is therefore a $5$-cycle. Let $I \subset \Lambda$ be a non-empty stable subset. The fact that $\xi$ and $\sigma(\xi)$ are adjacent for all $\xi$ and that $\sigma$ is a $5$-cycle implies $\Card \partial I > \Card I$. $\Lambda$ therefore has finite volume by the \hyperref[thm_principal]{main theorem}. 
\end{proof}

We summarise this in table \ref{table:finitude_pq}, which treats arbitrary simplices (left) and ideal ones (right). An entry $(p,q)$ contains $f$ if every simplex of $\H^{p,q}$ has finite volume, $\infty$ if there exist simplices of infinite volume. 

\begin{table}[h]
	\centering
	\[\begin{array}{c}
	\begin{array}{c|ccc}
	& p=0 & p=1 & p \geq 2 \\
	\hline
	q=0 & \infty & \infty & f \\
	q \geq 1 & f & \infty & \infty \\
\end{array}
		\qquad
		\begin{array}{c|ccc}
			& p=2 & p\geq3 \\
			\hline
			q=0 & f & f \\
			q=1 & \infty & \infty \\
			q=2 & f & \infty \\
			q\geq3 & ? & \infty
		\end{array}
	\end{array}\]
	\caption{existence of simplices of infinite volume in $\H^{p,q}$ (left) and existence of ideal simplices of infinite volume in $\H^{p,q}$ (right)}
	\label{table:finitude_pq}
\end{table}

Several questions remain. We do not know whether there exists an upper bound on the volume of ideal simplices of $\H^{2,2}$. A fortiori, we do not know whether there exists a maximum, nor where it would be attained : a good candidate would be the lightlike pentagon in $\H^{2,2}$ (example \ref{exemple:pentagone}). Finally, given an integer $n \geq 3$, we do not know whether every ideal simplex of $\H^{2,n}$ has finite volume. Finally, let us state a differential-geometric corollary which is the original reason for this article. 
\begin{corollaire}
Let $S$ be a space-like surface of $\H^{2,2}$, complete, maximal and of finite total curvature. The volume of $\Conv(S)$ is finite. 
\end{corollaire}
\begin{proof}
By a theorem of Moriani (voir \cite{Mor25}), there exists a lightlike polygon $\Lambda_0$ and a \positif lift $s$ with convex hull $\Conv(s) = \Conv(S)$. Carathéodory's theorem asserts that $\Conv(s)$ is the (finite) union of the $\Conv s_{|\Lambda}$, for $\Lambda \subset \Lambda_0$ of cardinal at most $5$, which have finite volume by Corollary \ref{corollaire:simplexe_H22_fini}. 
\end{proof}

\subsection{Examples of graphs}\label{subsection:exemples}\,\\

\etape*[Ideal simplex of $\H^p$]
Let $\Lambda$ an ideal simplex of $\H^p$. The graph $\cal{G}(\Lambda)$ has $p+1$ vertices and all edges, except loops (see figure \ref{figure:graphe_exemple_1}). The volume $\Vol(\Lambda)$ of the convex hull is finite. To see this, one may apply Theorem \ref{thm_principal} : a non-empty stable subset $I \subset \Lambda$ contains a single vertex, so we have $\Card \partial I = p$, which is strictly greater than $\Card I = 1$ whenever $p \geq 2$.
\begin{figure}[hp]
	\centering
	\begin{tikzpicture}[scale=2]
		\tikzstyle{vertex}=[circle, fill=black, inner sep=2pt]
		\def\n{5}
		\foreach \i in {1,...,\n} {
			\node[vertex] (v\i) at ({360/\n * (\i-1)}:1) {};}
		\foreach \i in {1,...,\n} {
			\foreach \j in {1,...,\n} {
				\ifnum\i<\j
				\draw (v\i) -- (v\j);
				\fi}}
	\end{tikzpicture}
	\caption{Ideal simplex in $\H^{p}$ ($p=4$)}\label{figure:graphe_exemple_1}
\end{figure}
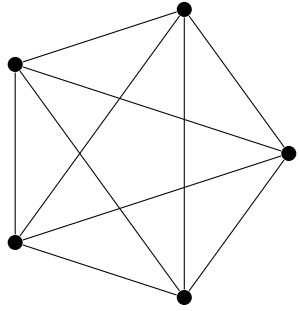

\etape*[$p$-crown in $\H^{p,p-1}$]\label{étape:exemple_couronne}
Let us consider a graph $\cal{G}$ with no loop and $2p$ vertices, each being adjacent to a unique other (see figure \ref{figure:graphe_exemple_2}). By Proposition \ref{proposition:classification_cycles}, $\cohom^1(\cal{G},\R^*)$ is reduced to a point, of signature $(p,p)$ : this graph is therefore realised only by a simplex $\Lambda_p$ of $\H^{p,p-1}$ (modulo isometric embedding of $\H^{p,p-1}$ into another pseudo-hyperbolic space). $\Lambda_p$ was defined in \cite{Barbot13} for $p=2$ and \cite{BK25} in general, where it is called a $p$-crown. $\Lambda$ has $2^{p-1}$ convex hulls (remark \ref{remarque:nombre_enveloppes_convexes}) and $\R^p$ acts by isometries on each convex hull $\cal{C}$. This action preserves the unique complete maximal $p$-submanifold contained in $\cal{C}$, and is simply transitive on this $p$-submanifold. The volume of $\cal{C}$ in $\H^{p,p-1}$ is therefore infinite. This also follows from the \hyperref[thm_principal]{main theorem}: every non-empty stable set of vertices $I$ contains at most one vertex in each connected component and satisfies $\Card \partial I = \Card I$.
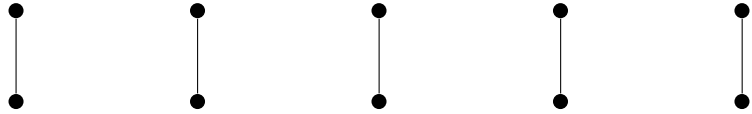
\begin{figure}[hp]
	\centering
	\begin{tikzpicture}[scale=1.2]
		\tikzstyle{vertex}=[circle, fill=black, inner sep=2pt]
		\def\n{5}
		\foreach \i in {1,...,\n} {
			\node[vertex] (a\i) at (2*\i,0) {};
			\node[vertex] (b\i) at (2*\i,1) {};
			\draw (a\i) -- (b\i);}
	\end{tikzpicture}
	\caption{$2p$-crown in $\H^{p,p-1}$ ($p=5$)}\label{figure:graphe_exemple_2}
\end{figure}

\etape*[Lightlike pentagon in $\H^{2,2}$]\label{exemple:pentagone}
We consider the $5$-cycle $C_5$ (see figure \ref{figure:graphe_exemple_3}). By Proposition \ref{proposition:classification_cycles}, $\cohom^1(\cal{G},\R)$ is reduced to a point, of signature $(2,3)$ : this graph is therefore realised by a unique simplex $\Lambda$, which is a simplex of $\H^{2,2}$ (modulo isometric embedding of $\H^{2,2}$ into another pseudo-hyperbolic space). $\Lambda$ is the lightlike pentagon of $\H^{2,2}$ and admits a unique convex hull $\cal{C}$ (remark \ref{remarque:nombre_enveloppes_convexes}), which contains a unique complete maximal surface. The volume of the convex hull is finite by the \hyperref[thm_principal]{main theorem}. Indeed, if $I$ is a non-empty stable set of vertices, then either $\Card I = 1$ and $\Card \partial I = 2$, or $\Card I = 2$ and $\Card \partial I = 3$.
\begin{figure}[hp]
	\centering
	\begin{tikzpicture}[scale=2]
		\tikzstyle{vertex}=[circle, fill=black, inner sep=2pt]
		\foreach \i in {1,...,5}{
			\node[vertex] (v\i) at ({72*(\i-1)+90}:1) {};}
		\draw (v1)--(v3)--(v5)--(v2)--(v4)--(v1);
	\end{tikzpicture}
	\caption{Lightlike pentagon in $\H^{2,2}$}\label{figure:graphe_exemple_3}
\end{figure}
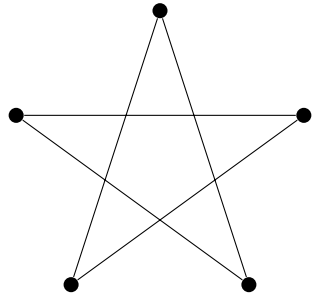

\etape*[Non-ideal simplex of infinite volume in $\H^{2,2}$]\label{étape:simplexe_non_idéal_infini}
This graph (figure \ref{figure:graphe_exemple_5}) is that of a simplex of $\H^{2,2}$, non-ideal and of infinite volume, constructed in the proof of Proposition \ref{proposition:existence_simplexes_infinis}. 
\begin{figure}[hp]
\centering
\begin{tikzpicture}[scale=1.5]
	\tikzstyle{vertex}=[circle, fill=black, inner sep=2pt]
	\node[vertex] (v1) at (0,1) {};
	\node[vertex] (v2) at (0,0) {};
	\draw (v1) -- (v2);
	\node[vertex] (v3) at (2,1) {};
	\node[vertex] (v4) at (2,0) {};
	\draw (v3) -- (v4);
	\node[vertex] (v5) at (4,0.5) {};
	\draw (v5) to[out=60,in=120,looseness=12] (v5);
\end{tikzpicture}
\caption{Non-ideal simplex of infinite volume in $\H^{2,2}$}\label{figure:graphe_exemple_5}
\end{figure}
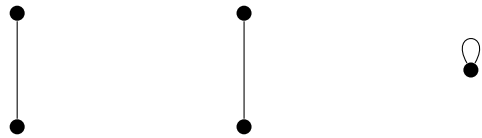

\FloatBarrier

\bibliographystyle{alpha}
\bibliography{bibli}

\end{document}